\documentclass[a4paper,11pt]{amsart}

\usepackage{amsmath, amssymb, amsthm}

\newcommand{\bZ}{\mathbb{Z}}
\newcommand{\bQ}{\mathbb{Q}}
\newcommand{\bC}{\mathbb{C}}
\newcommand{\sP}{\mathsf{P}}
\newcommand{\bP}{\mathbb{P}}
\newcommand{\bN}{\mathbb{N}}
\newcommand{\bF}{\mathbb{F}}
\newcommand{\Id}{\mathrm{Id}}
\newcommand{\cO}{\mathcal{O}}
\newcommand{\cS}{\mathcal{S}}

\newcommand{\sH}{\mathsf{H}}
\newcommand{\sF}{\mathsf{F}}
\newcommand{\sJ}{\mathsf{J}}

\newcommand{\can}{\operatorname{can}}

\newcommand{\PGL}{\mathrm{PGL}}
\newcommand{\supp}{\operatorname{supp}}
\newcommand{\bR}{\mathbb{R}}
\newcommand{\Res}{\operatorname{Res}}

\newcommand{\ord}{\operatorname{ord}}

\numberwithin{equation}{section}
\theoremstyle{plain}
\newtheorem{theorem}{Theorem}[section]
\newtheorem{lemma}[theorem]{Lemma}

\newtheorem{mainth}{Theorem}

\theoremstyle{definition}
\newtheorem{definition}[theorem]{Definition}
\newtheorem{notation}[theorem]{Notation}
\newtheorem*{convention}{Convention}
\newtheorem*{question}{Question}
\newtheorem*{acknowledgement}{Acknowledgement}
\theoremstyle{remark}
\newtheorem{remark}[theorem]{Remark}
\newtheorem{observation}[theorem]{Observation}

\begin{document}
\title[On a characterization of polynomials]{
On a characterization of polynomials among rational functions
in non-archimedean dynamics}
\author[Y\^usuke Okuyama]{Y\^usuke Okuyama}
\address{
Division of Mathematics,
Kyoto Institute of Technology, Sakyo-ku,
Kyoto 606-8585 Japan}
\email{okuyama@kit.ac.jp}
\author[Ma{\l}gorzata Stawiska]{Ma{\l}gorzata Stawiska}
\address{Mathematical Reviews, 416 Fourth St., Ann Arbor, MI 48103, USA}
\email{stawiska@umich.edu}

\date{\today}


\begin{abstract}
 We study a question on characterizing polynomials
 among rational functions of degree $>1$ on the projective line
 over an algebraically closed field that is complete with respect to 
 a non-trivial and non-archimedean absolute value,
 from the viewpoint of dynamics and potential theory on the Berkovich
 projective line.
\end{abstract}

\maketitle

\section{Introduction}\label{sec:intro}

Let $K$ be an algebraically closed field
that is complete with respect to a non-trivial and non-archimedean 
absolute value $|\cdot|$. 
The {\em Berkovich} projective line $\sP^1=\sP^1(K)$
is, as a topological augmentation of the (classical) 
projective line $\bP^1=\bP^1(K)=K\cup\{\infty\}$,
a compact, locally compact, uniquely arcwise connected, and
Hausdorff topological space. The set $\sH^1:=\sP^1\setminus\bP^1$
is called the Berkovich upper half space in $\sP^1$.

Let $f\in K(z)$ be a rational function of degree $d>1$. 
For every $n\in\bN$, set $f^n:=f\circ f^{n-1}$, where $f^0:=\Id_{\bP^1}$.
The action of $f$ on $\bP^1$ uniquely extends to a continuous endomorphism
on $\sP^1$, which is still open, surjective, and fiber-discrete, and
preserves both $\bP^1$ and $\sH^1$.
Let us define 
the {\em Berkovich} Julia set $\sJ(f)$ of $f$
by the set of all points $\cS\in\sP^1$
such that for any open neighborhood $U$ of $\cS$ in $\sP^1$,
\begin{gather*}
 \sP^1\setminus E(f)\subset\bigcup_{n\in\bN}f^n(U),
\end{gather*}
where the set $E(f):=\{a\in\bP^1:\#\bigcup_{n\in\bN}f^{-n}(a)<\infty\}$
is called the (classical)
exceptional set of $f$ and is at most countable subset in $\bP^1$.
The local degree function $\deg_{\, \cdot}f$ on $\bP^1$ also 
canonically extends to $\sP^1$, and this extended local degree function
$\deg_{\,\cdot}(f)$ induces a canonical pullback operator $f^*$ from the space
of all Radon measures on $\sP^1$ to itself (see \S\ref{th:extension} below).
Corresponding to the construction of 
the unique maximal entropy measure in complex dynamics
(studied since Lyubich \cite{Lyubich83}, Freire--Lopes--Ma\~n\'e \cite{FLM83}, 
Ma\~n\'e \cite{Mane83}),
the $f$-{\em canonical measure} $\mu_f$ on $\sP^1$
has been constructed as the unique probability Radon measure 
$\nu$ on $\sP^1$ such that 
\begin{gather*}
 f^*\nu=d\cdot\nu\text{ on }\sP^1\quad\text{and that}\quad \nu(E(f))=0,
\end{gather*}
so in particular $\mu_f$ is invariant under $f$ in that
$f_*\mu_f=\mu_f$ on $\sP^1$.
The support of $\mu_f$ 
coincides with $\sJ(f)$ and is the minimal non-empty
and closed subset in $\sP^1$ backward invariant under $f$
(\cite{FR09}). The {\em Berkovich} Fatou set of $f$ is defined by
\begin{gather*}
 \sF(f):=\sP^1\setminus\sJ(f), 
\end{gather*}
and each component of $\sF(f)$ is called a {\em Berkovich Fatou component} of $f$.
We note that $E(f)\subset\sF(f)$.
A Berkovich Fatou component of $f$ is mapped properly to a
Berkovich Fatou component of $f$ under $f$, and the preimage of a Berkovich
Fatou component of $f$ under $f$
is the union of at most $d$ Berkovich Fatou components of $f$.

\begin{definition}
For every $z\in\sF(f)\cap\bP^1$,
let $D_z=D_z(f)$ be the Berkovich Fatou component of $f$
containing $z$.
\end{definition}

For any $z\in\sF(f)\cap\bP^1$,
the compact subset $\sP^1\setminus D_z$ in $\sP^1$
is of logarithmic capacity $>0$ with pole $z$, or equivalently,
there is the unique {\em equilibrium mass distribution}
$\nu_{z,\sP^1\setminus D_z}$ on $\sP^1\setminus D_z$
with pole $z$, which is in fact supported by $\partial D_z\subset\sJ(f)$
(we will recall some details on the logarithmic potential theory 
on $\sP^1$ in \S\ref{sec:logarithmic} below). 
If $f(\infty)=\infty\in\sF(f)$,
then $\nu_{\infty,\sP^1\setminus D_\infty}$ is invariant under $f$ in that
\begin{gather*}
 f_*(\nu_{\infty,\sP^1\setminus D_\infty})=\nu_{\infty,\sP^1\setminus D_\infty}
\quad\text{on }\sP^1
\end{gather*}
(see Lemma \ref{th:invharmonic} below).
If moreover $f\in K[z]$ or equivalently 
$f^{-1}(\infty)=\{\infty\}$, 
then $\infty\in E(f)$, $f^{-1}(D_\infty)=D_\infty$, and we can see 
\begin{gather*}
 \mu_f=\nu_{\infty,\sP^1\setminus D_{\infty}}\quad\text{on }\sP^1
\end{gather*}
(since Brolin \cite{Brolin} in complex dynamics). Let $\delta_{\cS}$
be the Dirac measure on $\sP^1$ at $\cS\in\sP^1$.

Our aim is to study 
whether polynomials can be characterized among rational functions
of degree $>1$ using potential theory in non-archimedean setting, 
corresponding to 
the studies \cite{ObaPitcher,Lopes86,Lalley92,MR92,OS11,OS18} in complex dynamics.
Concretely, we study the following question
on a characterization of polynomials among rational functions
in non-archimedean dynamics.
\begin{question}
 Let $f\in K(z)$ be a rational function of degree $>1$, and
 suppose that $f(\infty)=\infty\in\sF(f)$ 
 (so in particular $f(D_\infty)=D_\infty$)
 and that $\sJ(f)\not\subset\sH^1$. 
 Then, are the statements
\begin{center}
  (i) $f\in K[z]$\quad and\quad
 (ii) $\mu_f=\nu_{\infty,\sP^1\setminus D_\infty}$ on $\sP^1$
\end{center} 
 equivalent? 
\end{question}
 The corresponding question in complex dynamics
 has been answered affirmatively (Lopes \cite{Lopes86}).

Here are a few comments on this Question.
We already mentioned that (i) implies (ii)
(without assuming $\sJ(f)\not\subset\sH^1$). It is not difficult
to construct such $f\in K(z)\setminus K[z]$ of degree $>1$
that $f(D_\infty)=D_\infty$, that
$f(\infty)\neq\infty\in\sF(f)$, that $\sJ(f)\not\subset\sH^1$, and that
$\mu_f=\nu_{\infty,\sP^1\setminus D_\infty}$ on $\sP^1$
(e.g.,\ Remark \ref{th:easy} below).
On the other hand,
if $\sJ(f)\subset\sH^1$, then for any $g\in K(z)$
of the same degree as that of $f$
which is close enough to $f$ (in the coefficients topology), both the Berkovich
Julia set $\sJ(g)$ of $g$ and the action of $g$ on $\sJ(g)$ are {\em same} as those of $f$ (cf.\ \cite[\S5.3]{FR09}). 
Since there is $f\in K[z]$ of degree $>1$
satisfying $\sJ(f)\subset\sH^1$ 
(e.g., such $f$ that has a potentially good reduction, see below
a characterization of this condition), 
for any such $f$ and any $b\in K$, if $0<|b|\ll 1$, then
$f_b(z):=f(z)/bz\in K(z)\setminus K[z]$ is of the same degree 
as that of $f$ and satisfies that
$f_b(\infty)=\infty\in\sF(f_b)$, that $\sJ(f_b)\subset\sH^1$, and
that $\mu_{f_b}=\nu_{\infty,\sP^1\setminus D_\infty(f_b)}$ on $\sP^1$.

Recall that $f$ {\em has a potentially good reduction} if and only if
there exists a point $\cS\in\sH^1$ such that
\begin{gather*}
 f^{-1}(\cS)=\{\cS\};
\end{gather*}
then $\sJ(f)=\{\cS\}(\subset\sH^1$ so $\infty\in\sF(f)$) 
and $\mu_f=\nu_{\infty,\sP^1\setminus D_\infty}=\delta_{\cS}$ 
on $\sP^1$ (see also Remark \ref{th:potgood} below). We say $f$ has no potentially good reductions
if $f$ does not have a potentially good reduction.

We already mentioned that 
the total invariance $f^{-1}(D_\infty)=D_\infty$ of $D_\infty$
under $f$ is a necessary condition for $f\in K[z]$. 
Our first result is the following more general statement, 
under no potentially good reductions.

\begin{mainth}\label{th:polynomial}
Let $K$ be an algebraically closed field
that is complete with respect to a non-trivial and non-archimedean 
absolute value.
Let $f\in K(z)$ be a rational function of degree $>1$. 
If $\infty\in\sF(f)$, $f(D_\infty)=D_\infty$,
$\mu_f=\nu_{\infty,\sP^1\setminus D_\infty}$ on $\sP^1$,
and $f$ has no potentially good reductions, then 
\begin{gather*}
 f^{-1}(D_\infty)=D_\infty.
\end{gather*}
\end{mainth}

Our second result is that even if we assume 
in addition $\sJ(f)\subset\bP^1$, 
the latter statement (ii) does not necessarily imply the former (i) in Question. 

Pick a prime number $p$.
The $p$-adic norm $|\cdot|_p$ on $\bQ$ 
is normalized so that for any $m,\ell\in\bZ\setminus\{0\}$
not divisible by $p$ and any $r\in\bZ$,
$\bigl|\frac{m}{\ell}p^r\bigr|_p=p^{-r}$.
The completion $\bQ_p$ of $(\bQ_p,|\cdot|_p)$ is still a field, and
the extended norm $|\cdot|_p$ on $\bQ_p$ extends to 
an algebraic closure $\overline{\bQ_p}$ of $\bQ_p$ as a norm.
The completion $\bC_p$ of $(\overline{\bQ_p},|\cdot|_p)$ is 
still an algebraically closed field, and 
the extended norm $|\cdot|_p$ on $\bC_p$ is a non-trivial
and non-archimedean absolute value on $\bC_p$.
The completion $\bZ_p$ of $(\bZ,|\cdot|_p)$ is a complete discrete valued 
local ring and has the 
unique maximal ideal $p\bZ_p$,
and coincides with the ring of $\bQ_p$-integers $\{z\in\bQ_p:|z|_p\le 1\}$.
In particular, the residual field of $\bQ_p$ is $\bF_p$.

The following 
counterexample of the implication (ii)$\Rightarrow$(i)
in Question is suggested to the authors by Juan Rivera-Letelier.

\begin{mainth}\label{th:counterexample}
 Pick a prime number $p$, and set 
\begin{gather*}
 f(z):=\frac{z^p-1}{p}\in\bQ[z]\quad\text{and}\quad
 A(z):=\frac{az+b}{cz+d}\in\PGL(2,\bZ_p).
\end{gather*} 
If $c\neq 0$ and $(a,b,c,d)$ is close enough to $(1,0,0,1)$ in $(\bZ_p)^4$, then 
 there is an attracting fixed point $z_A$ of $f\circ A$ in $\bC_p\setminus\bZ_p$
$($so $z_A\in\sF(f\circ A))$
 such that 
\begin{gather*}
 \sJ(f\circ A)=\bZ_p=\sP^1(\bC_p)\setminus D_{z_A}(f\circ A)\quad\text{and}\\
 \nu_{z_A,\bZ_p}=\nu_{\infty,\bZ_p}\quad\text{on }\sP^1(\bC_p). 
\end{gather*} 
Then setting $m_A(z):=1/(z-z_A)\in\PGL(2,\bC_p)$,
 the rational function 
 $g_A(z):=m_A\circ (f\circ A)\circ m_A^{-1}\in\bC_p(z)$ is of 
 degree $p$ and satisfies 
 $g_A\not\in\bC_p[z]$, $g_A(\infty)=\infty\in\sF(g_A)$,
 $\sJ(g_A)\subset\bP^1(\bC_p)$, and
\begin{gather*}
 \mu_{g_A}=\nu_{\infty,\sP^1(\bC_p)\setminus D_\infty(g_A)}\quad\text{on }\sP^1(\bC_p).
\end{gather*}
\end{mainth}

\subsection{Organization of this article}
In Sections \ref{sec:facts} and \ref{sec:dynamics},
we prepare  background material from potential theory and dynamics, respectively.
In Section \ref{sec:computation}, we make 
preparatory computations from
potential theory and give a proof of the invariance of $\nu_{\infty,\sP^1\setminus D_\infty}$ under $f$ when $f(\infty)=\infty\in\sF(f)$. 
In Sections \ref{sec:proof} and \ref{sec:counterexample},
we show Theorems \ref{th:polynomial} and \ref{th:counterexample}, respectively.

\section{Background from potential theory on $\sP^1$}\label{sec:facts}

Let $K$ be an algebraically closed field that is
complete with respect to a non-trivial and non-archimedean
absolute value $|\cdot|$; in general,
a norm $|\cdot|$ on a field $k$ is non-trivial
if $|k|\not\subset\{0,1\}$, and is non-archimedean if $|\cdot|$
satisfies the strong triangle inequality
\begin{gather*}
 |x+y|\le\max\{|x|,|y|\}\quad\text{for any }x,y\in k.
\end{gather*}
For the foundation of potential theory on $\sP^1=\sP^1(K)$, see 
\cite[\S 5, \S 8]{BR10}, \cite[\S 7]{FJbook},
\cite[\S 3]{FR06}, \cite{Thuillierthesis},
and the survey \cite[\S 1-\S 4]{Jonsson15} and the book \cite[\S 13]{BenedettoBook}.
In what follows, we adopt a presentation from \cite[\S 2, \S 3]{OkuDivisor}.   

\begin{notation}
 Let 
\begin{gather*}
 \pi:K^2\setminus\{(0,0)\}\to\bP^1=\bP^1(K)=K\cup\{\infty\} 
\end{gather*} 
be the canonical projection such that
\begin{gather*}
 \pi(p_0,p_1)=
\begin{cases}
 p_1/p_0 &\text{if }p_0\neq 0,\\
 \infty &\text{if }p_0=0,
\end{cases}
\end{gather*} 
following the convention on coordinate of $\bP^1$ from the book \cite{FvdP04}.

On $K^2$, let $\|(p_0,p_1)\|$ be the maximum norm
 $\max\{|p_0|,|p_1|\}$.
 With the wedge product
 $(p_0,p_1)\wedge(q_0,q_1):=p_0q_1-p_1q_0$ on $K^2$,
 the normalized chordal metric $[z,w]$ on $\bP^1$ is 
 the function
 \begin{gather*}
 [z,w]:=\frac{|p\wedge q|}{\|p\|\cdot\|q\|}(\le 1)
 \end{gather*}
 on $\bP^1\times\bP^1$, where $p\in\pi^{-1}(z),q\in\pi^{-1}(w)$.
\end{notation}

\subsection{Berkovich projective line $\sP^1$}\label{th:berkovich}
A ($K$-closed) {\em disk in} $K$
is a subset in $K$ written as $\{z\in K:|z-a|\le r\}$ for some $a\in K$ and some $r\ge 0$.
By the strong triangle inequality, 
two decreasing infinite sequences of disks in $K$ either
{\em infinitely nest} or {\em are eventually disjoint}. This alternative
induces the {\em cofinal} equivalence relation 
among decreasing (or more precisely, nesting and non-increasing)
infinite sequences of disks in $K$, and
the set of all cofinal equivalence classes $\cS$ of
decreasing infinite sequences $(B_n)$ of disks in $K$ together with $\infty\in\bP^1$ 
is, as a set, nothing but $\sP^1$ (\cite[p.\ 17]{Berkovichbook}); 
if $B_{\cS}:=\bigcap_n B_n\neq\emptyset$, then $B_{\cS}$ is itself a disk in $K$,
and we also say $\cS$ is represented by $B_{\cS}$.
For example, the {\em canonical $($or Gauss$)$ point} $\cS_{\can}$ in $\sP^1$ is
represented by the the ring of $K$-integers
\begin{gather*}
 \mathcal{O}_K:=\{z\in K:|z|\le 1\}, 
\end{gather*}
and each $z\in K$ 
is represented by the disk $\{z\}$ in $K$. The above alternative between two 
(decreasing infinite sequences of) disks in $K$
also induces a canonical ordering $\succeq$ on $\sP^1$ so that
$\infty$ is the unique maximal element in $(\sP^1,\succeq)$ and that
for every $\cS,\cS'\in\sP^1\setminus\{\infty\}$ 
satisfying $B_{\cS},B_{\cS'}\neq\emptyset$, 
$\cS\succeq\cS'$ iff $B_{\cS}\supset B_{\cS'}$
(the description of $\succeq$
is a little complicated unless $B_{\cS},B_{\cS'}\neq\emptyset$),
and equips $\sP^1$ with a (profinite) tree structure. 
The topology of $\sP^1$ coincides with the weak (or observer) topology
on $\sP^1$ as a (profinite) tree,
so that $\sP^1$ is compact and 
uniquely arcwise-connected, and contains both
$\bP^1$ and $\sH^1$ as dense subsets.
For the details on the tree structure on $\sP^1$, 
see e.g.\ Jonsson \cite[\S 2]{Jonsson15}.

\subsection{Action of rational functions on $\sP^1$}\label{th:extension}
Let $h\in K(z)$ be a rational function. The action
of $h$ on $\bP^1$ uniquely extends to a continuous endomorphism on $\sP^1$.
Suppose in addition that $\deg h>0$. Then
the extended action of $h$ on $\sP^1$ is surjective and open, has discrete 
(so finite) fibers,
and preserves both $\bP^1$ and $\sH^1$, and the {\em local degree} function
$z\mapsto\deg_zh$ on $\bP^1$ also canonically extends to $\sP^1$ so that
for every $\cS\in\sP^1$,
\begin{gather*}
 \sum_{\cS'\in h^{-1}(\cS)}\deg_{\cS'}h=\deg h. 
\end{gather*}
The action of $h$ on $\sP^1$ induces the push-forward operator $h_*$
on the space of all continuous functions on $\sP^1$ to itself 
and, by duality, also the pullback operator $h^*$ 
on the space of all Radon measures on $\sP^1$ to itself; 
for every continuous test function $\phi$ on $\sP^1$,
$(h_*\phi)(\cdot)=\sum_{\cS'\in h^{-1}(\cdot)}(\deg_{\cS'}h)\cdot\phi(\cS')$
on $\sP^1$, and for every $\cS\in\sP^1$, $h^*\delta_{\cS}
=\sum_{\cS'\in h^{-1}(\cS)}(\deg_{\cS'}h)\cdot\delta_{\cS'}$
on $\sP^1$. For more details, see \cite[\S9]{BR10}, \cite[\S2.2]{FR09}.

\subsection{Kernel functions and the Laplacian on $\sP^1$}\label{th:kernels}
The {\em generalized Hsia kernel} $[\cS,\cS']_{\can}$ on $\sP^1$
{\em with respect to} $\cS_{\can}$ is a unique
upper semicontinuous and separately continuous extension
of the chordal distance function $\bP^1\times\bP^1\ni(z,z')\mapsto[z,z']$
to $\sP^1\times\sP^1$. 

More generally,
for every $z_0\in\bP^1$, the {\em generalized Hsia kernel}
\begin{gather*}
 [\cS,\cS']_{z_0}
:=
\begin{cases}
\displaystyle \frac{[\cS,\cS']_{\can}}{[\cS,z_0]_{\can}\cdot[\cS',z_0]_{\can}} & \text{on }(\sP^1\setminus\{z_0\})\times(\sP^1\setminus\{z_0\})\\ 
 +\infty & \text{on }(\{z_0\}\times\sP^1)\cup(\sP^1\times\{z_0\})
\end{cases}
\end{gather*}
on $\sP^1$ {\em with respect to} $z_0$
is a unique upper semicontinuous and separately 
continuous extension of the function $(\bP^1\setminus\{z_0\})\times(\bP^1\setminus\{z_0\})\ni(z,z')\mapsto[z,z']/([z,z_0]\cdot[z',z_0])$ as
a function $\sP^1\times\sP^1\to[0,+\infty]$.
In particular, the function
\begin{gather*}
 |\cS-\cS'|_{\infty}:=[\cS,\cS']_\infty
\end{gather*}
on $\sP^1\times\sP^1$
extends the distance function $K\times K\ni(z,z')\mapsto|z-z'|$ 
to $(\sP^1\setminus\{\infty\})\times(\sP^1\setminus\{\infty\})$,
jointly upper semicontinuously and separately continuously, and the function
\begin{gather*}
 |\cS|_\infty:=|\cS-0|_\infty(=[\cS,0]_\infty)\quad\text{on }\sP^1
\end{gather*}
extends the norm function $K\ni z\mapsto|z|$ to $\sP^1\setminus\{\infty\}$
continuously (see \cite[\S 3.4]{FR06}, \cite[\S 4.4]{BR10}). 

Let $\Omega_{\can}$ be the Dirac measure  
$\delta_{\cS_{\can}}$ on $\sP^1$ at $\cS_{\can}$.
The Laplacian $\Delta$ on $\sP^1$ is normalized so that for each $\cS'\in\sP^1$, 
\begin{gather*}
 \Delta\log[\cdot,\cS']_{\can}=\delta_{\cS'}-\Omega_{\can}
\end{gather*}
on $\sP^1$, and then, for every $z_0\in\bP^1$ and every $\cS'\in\sP^1\setminus\{z_0\}$, $\Delta\log[\cdot,\cS']_{z_0}=\delta_{\cS'}-\delta_{z_0}$ on $\sP^1$.
For the details on the construction and properties of $\Delta$, 
see \cite[\S 5]{BR10}, \cite[\S7.7]{FJbook}, \cite[\S2.4]{FR09}, 
\cite[\S 3]{Thuillierthesis}; in \cite{BR10,Thuillierthesis},
the opposite sign convention for $\Delta$ is adopted.

\subsection{Logarithmic potential theory on $\sP^1$}\label{sec:logarithmic}
For every $z\in\sP^1$ and
every positive Radon measure $\nu$ on $\sP^1$ supported by $\sP^1\setminus\{z\}$, 
the {\em logarithmic potential} of $\nu$ on $\sP^1$ with 
pole $z$ is the function
\begin{gather*}
 p_{z,\nu}(\cdot):=\int_{\sP^1}\log[\cdot,\cS']_{z}\nu(\cS')\quad\text{on }\sP^1,
\end{gather*}
and the {\em logarithmic energy} of $\nu$ with pole $z$ is defined by 
\begin{gather*}
 I_{z,\nu}:=\int_{\sP^1}p_{z,\nu}\nu\in[-\infty,+\infty).
\end{gather*}
Then $p_{z,\nu}:\sP^1\to[-\infty,+\infty]$ is upper semicontinuous,
and in fact is {\em strongly} upper semicontinuous in that for every $\cS\in\sP^1$,
\begin{gather}
 \limsup_{\cS'\to\cS}p_{z,\nu}(\cS')=p_{z,\nu}(\cS)\label{eq:susc}
\end{gather}
(\cite[Proposition 6.12]{BR10}).

For every non-empty
subset $C$ in $\sP^1$ and every $z\in\bP^1\setminus C$,
we say $C$ is {\em of logarithmic capacity $>0$ with pole $z$} if
\begin{gather*}
 V_z(C):=\sup_{\nu}I_{z,\nu}>-\infty,
\end{gather*}
where $\nu$ ranges over all probability Radon measures on $\sP^1$
supported by $C$; otherwise, we say $C$ is {\em of logarithmic capacity $0$
with pole $z$}. For every non-empty compact subset $C$ in $\sP^1$
of logarithmic capacity $>0$ with pole $z\in\sP^1\setminus C$, 
there is a {\em unique} probability Radon measure 
$\nu$ on $\sP^1$, which is called 
the {\em equilibrium mass distribution on $C$
with pole} $z$ and is denoted by $\nu_{z,C}$, 
such that $\supp\nu\subset C$ and that $I_{z,\nu}=V_z(C)$, and then
(i) $\nu_{z,C}(E)=0$ for any subset $E$ in $C$ of logarithmic capacity $0$
with pole $z$, 
(ii) 
letting $D_z$ be the component of
$\sP^1\setminus C$ containing $z$, we have 
\begin{gather*}
\supp\nu_{z,C}\subset\partial D_z,\quad
 p_{z,\nu_{z,C}}\ge I_{z,\nu_{z,C}}\text{ on }\sP^1,\quad
 p_{z,\nu_{z,C}}>I_{z,\nu_{z,C}}\text{ on }D_z,\quad\text{and}\\
 p_{z,\nu_{z,C}}\equiv I_{z,\nu_{z,C}}\text{ on }\sP^1\setminus(D_z\cup E),
\end{gather*}
where $E$ is a possibly empty $F_{\sigma}$-subset in $\partial D_z$
of logarithmic capacity $0$ with pole $z$, 
(iii) if in addition $p_{z,\nu_{z,C}}$ is continuous on 
$\sP^1\setminus{\{z\}}$, 
then 
\begin{gather*}
\supp\nu_{z,C}=\partial D_z\quad\text{and}\quad
p_{z,\nu_{z,C}}\equiv I_{z,\nu_{z,C}}\text{ on }\sP^1\setminus D_z,
\end{gather*}
and (iv) for any probability Radon measure $\nu'$ supported by $C$, we have
\begin{gather}
 \inf_{\cS\in C}p_{z,\nu'}\le I_{z,\nu_{z,C}}\le\sup_{\cS\in C}p_{z,\nu'}\label{eq:non-equilibrium}
\end{gather}
(see \cite[\S 6.2, \S 6.3]{BR10}).

We list a few observations.

\begin{observation}\label{th:affine}
 For every $a\in K\setminus\{0\}$ and every $b\in K$, 
 setting $\ell(z):=az+b\in\PGL(2,K)$, 
 we have $\log|\ell(\cS)-\ell(\cS')|_\infty=\log|\cS-\cS'|_\infty+\log|a|$
 on $K\times K$, and in turn on $\sP^1\times\sP^1$. 
 In particular, for every non-empty compact subset $C$ in
 $\sP^1\setminus\{\infty\}$ of logarithmic capacity $>0$ with pole $\infty$,
 we have  $I_{\infty,\nu_{\infty,\ell(C)}}=I_{\infty,\nu_{\infty,C}}+\log|a|$
and $\ell_*(\nu_{\infty,C})=\nu_{\infty,\ell(C)}$ on $\sP^1$.
\end{observation}

\begin{observation}\label{th:involution}
 Since the involution $\iota(z)=1/z\in\PGL(2,\cO_K)$ 
 acts on $(\bP^1,[z,w])$ isometrically,
 for any $z_0\in\bP^1$, we have
 $[\iota(\cS),\iota(\cS')]_{\iota(z_0)}=[\cS,\cS']_{z_0}$ 
 on $\bP^1\times\bP^1$, and in turn on $\sP^1\times\sP^1$.
 Hence for any non-empty compact subset $C$ in $\sP^1$ and 
 any $z\in\sP^1\setminus C$,
 if $C$ is of logarithmic capacity $>0$ with pole $z$, then 
 $V_z(C)=V_{\iota(z)}(\iota(C))$ 
 and $\iota_*(\nu_{z,C})=\nu_{\iota(z),\iota(C)}$ on $\sP^1$.
\end{observation}

\begin{observation}\label{th:perburb}
For every $z\in\bP^1$, 
the strong triangle inequality 
$[\cS,\cS'']_{z}\le\max\{[\cS,\cS']_{z},[\cS',\cS'']_{z}\}$
for $\cS,\cS',\cS''\in\sP^1$ still holds
 (see \cite[Proposition 4.10]{BR10}). Hence
 for every non-empty compact subset $C$ in $\sP^1\setminus\{\infty\}$ and
 every $z\in\sP^1\setminus C$ so close to $\infty$ that
 $[z,\infty]<\inf_{\cS\in C}[\cS,z]_{\can}$, we have
 $[\cdot,\infty]_{\can}=[\cdot,z]_{\can}$ on $C$, 
 which yields $[\cS,\cS']_\infty=[\cS,\cS']_{z}$ on $C\times C$, so
 if in addition $C$ is of logarithmic capacity $>0$ with pole $\infty$, then 
 $V_{\infty}(C)=V_z(C)$ 
 and $\nu_{\infty,C}=\nu_{z,C}$ on $\sP^1$.
\end{observation}

\subsection{Potential theory with a continuous weight on $\sP^1$}
A {\em continuous weight $g$ on} $\sP^1$ is
a continuous function on $\sP^1$ such that 
\begin{gather*}
 \mu^g:=\Delta g+\Omega_{\can}
\end{gather*}
is a probability Radon measure on $\sP^1$. Then
$\mu^g$ has no atoms on $\bP^1$, or more strongly,
$\mu^g(E)=0$ for any subset $E$ in $\sP^1$ of logarithmic capacity $0$
with some (indeed any) point in $\sP^1\setminus E$.

For a continuous weight $g$ on $\sP^1$,
the {\em $g$-potential kernel} on $\sP^1$
(the negative of an Arakelov Green kernel function on $\sP^1$
relative to $\mu^g$ \cite[\S 8.10]{BR10}) is an upper semicontinuous
function
\begin{gather}
 \Phi_g(\cS,\cS'):=\log[\cS,\cS']_{\can}-g(\cS)-g(\cS')\quad\text{on }\sP^1\times\sP^1.\label{eq:weighted}
\end{gather}
For every Radon measure $\nu$ on $\sP^1$,
the {\em $g$-potential} of $\nu$ on $\sP^1$ is the function
\begin{gather*}
 U_{g,\nu}(\cdot):=\int_{\sP^1}\Phi_g(\cdot,\cS')\nu(\cS')\quad\text{on }\sP^1,
\end{gather*}
and the {\em $g$-energy} of $\nu$ is defined by  
\begin{gather*}
 I_{g,\nu}:=\int_{\sP^1}U_{g,\nu}\nu\in[-\infty,+\infty).
\end{gather*}
The {\em $g$-equilibrium energy
$V_g$ of $($the whole$)$ $\sP^1$} is the supremum of the {\em $g$-energy} functional
$\nu\mapsto I_{g,\nu}$,
where $\nu$ ranges over all probability Radon measures on $\sP^1$. Then
$V_g\in\bR$ since $I_{g,\Omega_{\can}}>-\infty$. 
As in the logarithmic potential theory presented in the previous subsection,
there is a unique probability Radon measure $\nu^g$ on 
$\sP^1$, which is called 
the $g$-{\em equilibrium mass distribution on} $\sP^1$, 
such that $I_{g,\nu^g}=V_g$. In fact
\begin{gather*}
 U_{g,\nu^g}\equiv V_g\quad\text{on }\sP^1
\quad\text{and}\quad\nu^g=\mu^g\quad\text{on }\sP^1
\end{gather*}
(see \cite[Theorem 8.67, Proposition 8.70]{BR10}).

A continuous weight $g$ on $\sP^1$ is 
a {\em normalized weight on} $\sP^1$ if
$V_g=0$. For a continuous weight $g$ on $\sP^1$,
$\overline{g}:=g+V_g/2$ is the unique normalized weight on $\sP^1$ 
satisfying $\mu^{\overline{g}}=\mu^g$.

\section{Background from Dynamics on $\sP^1$}\label{sec:dynamics}
Let $K$ be an algebraically closed field that is
complete with respect to a non-trivial and non-archimedean absolute value $|\cdot|$. 
For a potential-theoretic study of dynamics of a rational function
of degree $>1$ on $\sP^1=\sP^1(K)$,
see \cite[\S 10]{BR10}, \cite[\S 3]{FR09}, \cite[\S 5]{Jonsson15},
and \cite[\S 13]{BenedettoBook}.
In the following, we adopt a presentation from \cite[\S 8.1]{OkuDivisor}.

\subsection{Canonical measure and the dynamical Green function of $f$ on $\sP^1$}\label{sec:canonical}
Let $f\in K(z)$ be a rational function of degree $d>1$.
We call $F\in(K[p_0,p_1]_d)^2$ a {\em lift} of $f$ if 
\begin{gather*}
 \pi\circ F=f\circ\pi
\end{gather*}
on $K^2\setminus\{(0,0)\}$, where for each $j\in\bN\cup\{0\}$, $K[p_0,p_1]_j$
is the set of all homogeneous polynomials in $K[p_0,p_1]$ of degree $j$,
as usual. 
A lift $F=(F_0,F_1)$ of $f$ is unique up to multiplication in $K\setminus\{0\}$.
Setting $d_0:=\deg F_0(1,z)$ and $d_1:=\deg F_1(1,z)$
and letting $c^F_0,c^F_1\in K\setminus\{0\}$ be the
coefficients of the maximal degree terms of $F_0(1,z),F_1(1,z)\in K[z]$,
respectively, the {\em homogeneous} resultant
\begin{gather*}
 \Res F=
(c^F_0)^{d-d_1}\cdot(c^F_1)^{d-d_0}\cdot R\bigl(F_0(1,\cdot),F_1(1,\cdot)\bigr)\in K
\end{gather*}
of $F$ does not vanish,
where $R(P,Q)\in K$ is the usual resultant of $(P,Q)\in (K[z])^2$
(for the details on $\Res F$, see e.g.\ \cite[\S2.4]{SilvermanDynamics}).

Let $F$ be a lift of $f$, and for every $n\in\bN\cup\{0\}$, 
set $F^n=F\circ F^{n-1}$ where $F^0:=\Id_{K^2}$. Then for every $n\in\bN$, 
$F^n$ is a lift of $f^n$, 
and the function 
\begin{gather*}
 T_{F^n}:=\log\|F^n\|-d^n\cdot\log\|\cdot\|
\end{gather*}
on $K^2\setminus\{(0,0)\}$ descends to $\bP^1$ and in turn extends continuously to
$\sP^1$,
satisfying the equality $\Delta T_{F^n}=(f^n)^*\Omega_{\can}-d^n\cdot\Omega_{\can}$
on $\sP^1$ (see, e.g., \cite[Definition 2.8]{Okucharacterization}). 
The dynamical Green function of $F$ on $\sP^1$ is the uniform limit 
$g_F:=\lim_{n\to\infty}T_{F^n}/d^n$ on $\sP^1$,
which is a continuous weight on $\sP^1$.
The {\em energy} formula 
\begin{gather*}
 V_{g_F}=-\frac{\log|\Res F|}{d(d-1)}
\end{gather*}
is due to DeMarco \cite{DeMarco03}
for archimedean $K$ by a dynamical argument,
and due to Baker--Rumely \cite{BR06} when $f$ is defined over a number field;
see Baker \cite[Appendix A]{Baker09} or
the present authors \cite[Appendix]{OS11} for a simple and potential-theoretic
proof of this remarkable formula, for general $K$.
The $f$-{\em canonical measure} is the probability Radon measure
\begin{gather*}
 \mu_f:=\Delta g_F+\Omega_{\can}\quad\text{on }\sP^1.
\end{gather*}
The measure $\mu_f$ is independent of the choice of the lift $F$ of $f$,
has no atoms in $\bP^1$,
and satisfies the $f$-balanced property $f^*\mu_f=d\cdot\mu_f$ 
(so in particular $f_*\mu_f=\mu_f$) on $\sP^1$. For more details,
see \cite[\S10]{BR10}, \cite[\S2]{ChambertLoir06}, \cite[\S3.1]{FR09}.

The {\em dynamical Green function $g_f$ of $f$ on} $\sP^1$ is 
the unique normalized weight on $\sP^1$ such that $\mu^{g_f}=\mu_f$.
By the above energy formula on $V_{g_F}$ and 
\begin{gather*}
 \Res(cF)=c^{2d}\cdot \Res F\quad\text{for every }c\in K\setminus\{0\},
\end{gather*}
there is a lift $F$ of $f$ normalized so that $V_{g_F}=0$
or equivalently that $g_F=g_f$ on $\sP^1$, and
such a {\em normalized lift} $F$ of $f$
is unique up to multiplication in $\{z\in K:|z|=1\}$.
By $g_f=g_F=\lim_{n\to\infty}T_{F^n}/d^n$ on $\sP^1$
for a normalized lift $F$ of $f$,
for every $n\in\bN$, we have
$g_{F^n}=g_{f^n}=g_f$ on $\sP^1$ and  $\mu_{f^n}=\mu_f$ on $\sP^1$.
We note that
$g_f\circ f=d\cdot\lim_{n\to\infty}T_{F^{n+1}}/d^{n+1}-T_F=d\cdot g_f-T_F$
on $\bP^1$, that is,
\begin{gather}
 d\cdot g_f-g_f\circ f=T_F\label{eq:invarianceGreen}
\end{gather}
on $\bP^1$, and in turn on $\sP^1$ by the density of $\bP^1$ in $\sP^1$
and the continuity of both sides on $\sP^1$
(cf.\ \cite[Proof of Lemma 2.4]{OkuLyap}).

\subsection{Fundamental properties of $\mu_f$}\label{sec:support}
Recall the definition of $\sJ(f)$ in Section \ref{sec:intro}.
The characterization of $\mu_f$ as the unique probability Radon measure 
$\nu$ on $\sP^1$ such that $\nu(E(f))=0$ and that
$f^*\nu=d\cdot \nu$ on $\sP^1$ 
is a consequence of the following equidistribution theorem;
{\em for every probability Radon measure $\mu$ on $\sP^1$, if $\mu(E(f))=0$, then}
\begin{gather}
 \lim_{n\to\infty}\frac{(f^n)^*\mu}{d^n}=\mu_f\quad \text{\em weakly on }\sP^1.\label{eq:equidist}
\end{gather}
This foundational result is due to Favre--Rivera-Letelier \cite{FR09} 
(for a purely potential-theoretic proof, see also Jonsson \cite{Jonsson15})
and is a non-archimedean counterpart to
Brolin \cite{Brolin}, Lyubich \cite{Lyubich83},
Freire--Lopes--Ma\~n\'e \cite{FLM83}. 

\begin{remark}\label{th:exceptional}
The {\em classical Julia set} $\sJ(f)\cap\bP^1$
of $f$ coincides with the set of all points in $\bP^1$ at each of which the family
$\bigl(f^n:(\bP^1,[z,w])\to(\bP^1,[z,w])\bigr)_{n\in\bN}$ is not locally 
equicontinuous (see, e.g.,\ \cite[Theorem 10.67]{BR10}).
\end{remark}

The equality $\supp\mu_f=\sJ(f)$ holds; the inclusion $\sJ(f)\subset\supp\mu_f$
follows from the definition of $\sJ(f)$,
the balanced property $f^*\mu_f=d\cdot\mu_f$ on $\sP^1$, and 
$\supp\mu_f\not\subset E(f)$ (or more precisely,
recalling that $E(f)$ is an at most countable subset
in $\bP^1$ and that $\mu_f$ has no atoms in $\bP^1$). 
The opposite inclusion $\supp\mu_f\subset\sJ(f)$ follows from the definition of $\sJ(f)$ and 
the above equidistribution theorem.
\begin{remark}[see e.g.\ {\cite[Corollary 10.33]{BR10}}]\label{th:potgood}
 If $\mu_f$ has an atom in $\sP^1$, then $f$ has a potentially good reduction,
 so in particular $\sJ(f)$ is a singleton in $\sH^1$.
\end{remark}

For every $n\in\bN$, by $\supp\mu_f=\sJ(f)$ and $\mu_{f^n}=\mu_f$ on $\sP^1$, 
we also have $\sJ(f^n)=\sJ(f)$.
For every $m\in\PGL(2,K)$, we have
$m_*\mu_f=\mu_{m\circ f\circ m^{-1}}$ on $\sP^1$, 
$m(\sJ(f))=\sJ(m\circ f\circ m^{-1})$, and $m(\sF(f))=\sF(m\circ f\circ m^{-1})$.

\subsection{Root divisors on $\bP^1$
and the proximity functions on $\sP^1$}\label{sec:divisor}
For any distinct $h_1,h_2\in K(z)$,
let $[h_1=h_2]$ be the effective ($K$-)divisor on $\bP^1$
defined by all solutions to the equation $h_1=h_2$ in $\bP^1$
taking into account their multiplicities, which is also regarded as
the Radon measure 
\begin{gather*}
 \sum_{w\in\bP^1}(\ord_w[h_1=h_2])\cdot\delta_w
\end{gather*}
on $\sP^1$.
The function 
$\bP^1\ni z\mapsto[h_1(z),h_2(z)]$ between $h_1$ and $h_2$ uniquely extends to
a continuous function $\cS\mapsto[h_1,h_2]_{\can}(\cS)$ on $\sP^1$
(see, e.g., \cite[Proposition 2.9]{Okucharacterization}), so that
for every continuous weight $g$ on $\sP^1$, (the $\exp$ of) the function
\begin{gather}
 \Phi(h_1,h_2)_g(\cS):=\log[h_1,h_2]_{\can}(\cS)-g(h_1(\cS))-g(h_2(\cS))
\quad\text{on }\sP^1\label{eq:proximityweighted}
\end{gather}
is a unique continuous extension of (the $\exp$ of) the 
function $\bP^1\ni z\mapsto\Phi_g(h_1(z),h_2(z))$.

\section{Potential-theoretic computations}
\label{sec:computation}
Let $K$ be an algebraically closed field that is
complete with respect to a non-trivial and non-archimedean
absolute value $|\cdot|$.
Let $f\in K(z)$ be a rational function of degree $d>1$.

\begin{lemma}[Riesz's decomposition for the pullback of an atom]
 For every $\cS\in\sP^1$,
 \begin{gather}
  \Phi_{g_f}(f(\cdot),\cS)=U_{g_f,f^*\delta_\cS}(\cdot)
  \quad\text{on }\sP^1.\label{eq:Riesz}
 \end{gather}
\end{lemma}

\begin{proof}
Fix a lift $F$ of $f$ normalized so that $g_F=g_f$ on $\sP^1$.
Fix $w\in\bP^1$ and $W\in\pi^{-1}(w)$.
Choose a sequence $(q_j)_{j=1}^{d}$ 
in $K^2\setminus\{(0,0)\}$ such that $F(p_0,p_1)\wedge W\in K[p_0,p_1]_d$ factors as
$F(p_0,p_1)\wedge W=\prod_{j=1}^{d}((p_0,p_1)\wedge q_j)$ in $K[p_0,p_1]$,
which with \eqref{eq:invarianceGreen} implies
\begin{multline*}
 \Phi_{g_f}(f(\cdot),w)-U_{g_f,f^*\delta_w}(\cdot)\\
 \equiv -(g_f(w)+\log\|W\|)
 +\sum_{j=1}^{d}(g_f(\pi(q_j))+\log\|q_j\|)=:C
\end{multline*}
on $\bP^1$, and in turn on $\sP^1$ by
the density of $\bP^1$ in $\sP^1$ and
the continuity of (the $\exp$ of)
both sides on $\sP^1$.
Integrating both sides against $\mu_f$ over $\sP^1$,
since
$\int_{\sP^1}U_{g_f,f^*\delta_w}\mu_f
=\int_{\sP^1}U_{g_f,\mu_f}(f^*\delta_w)=0$ 
(by $U_{g_f,\mu_f}\equiv 0$) and $f_*\mu_f=\mu_f$, we have 
\begin{gather*}
 C=\int_{\sP^1}\Phi_{g_f}(f(\cdot),w)\mu_f
 =U_{g_f,f_*\mu_f}(w)=U_{g_f,\mu_f}(w)=0.
\end{gather*}
This completes the proof
of \eqref{eq:Riesz} in the case $\cS=w\in\bP^1$.

Fix $\cS\in\sH^1$. By the density of $\bP^1$ in $\sP^1$,
we can choose
a sequence 
$(w_n)$ in $\bP^1$ tending to $\cS$ as $n\to\infty$.
Then $\lim_{n\to\infty}f^*\delta_{w_n}=f^*\delta_\cS$ weakly on $\sP^1$ and,
for every $n\in\bN$, applying \eqref{eq:Riesz} to $\cS=w_n\in\bP^1$,
we have $\Phi_{g_f}(f(\cdot),w_n)=U_{g_f,f^*\delta_{w_n}}(\cdot)$
on $\sP^1$. Hence, for each $\cS'\in\sH^1$, by the continuity of
both $\Phi_{g_f}(f(\cS'),\cdot)$ and
$\Phi_{g_f}(\cS',\cdot)$ on $\sP^1$, we have
\begin{gather*}
 \Phi_{g_f}(f(\cS'),\cS)
=\lim_{n\to\infty}\Phi_{g_f}(f(\cS'),w_n)
=\lim_{n\to\infty}U_{g_f,f^*\delta_{w_n}}(\cS')
=U_{g_f,f^*\delta_{\cS}}(\cS').
\end{gather*}
This completes the proof of \eqref{eq:Riesz}
by the density of $\sH^1$ in $\sP^1$
and the continuity of (the $\exp$ of) 
both $\Phi_{g_f}(f(\cdot),\cS)$ and $U_{g_f,f^*\delta_{\cS}}(\cdot)$ on $\sP^1$.
\end{proof}

The following computation 
is an application of Lemma \ref{eq:Riesz}.
We include a proof of it although it will not be used in this article. 

\begin{lemma}[Riesz's decomposition for the fixed points divisor on $\bP^1$]\label{th:proximity}
\begin{gather}
\Phi(f,\Id_{\bP^1})_{g_f}=U_{g_f,[f=\Id_{\bP^1}]}\quad\text{on }\sP^1.\label{eq:Rieszfixed}
 \end{gather}
\end{lemma}

\begin{proof}
Fix a lift $F$ of $f$ normalized so that $g_F=g_f$ on $\sP^1$.
Choose a sequence $(q_j)_{j=1}^{d+1}$ in $K^2\setminus\{(0,0)\}$ so that 
$(F\wedge\Id_{\bP^1})(p_0,p_1)\in K[p_0,p_1]_{d+1}$ factors as
$(F\wedge\Id_{\bP^1})(p_0,p_1)=\prod_{j=1}^{d+1}((p_0,p_1)\wedge q_j)$ in $K[p_0,p_1]$, 
which with \eqref{eq:invarianceGreen} implies
\begin{gather*}
 \Phi(f,\Id_{\bP^1})_{g_f}-U_{g_f,[f=\Id_{\bP^1}]}\equiv
 \sum_{j=1}^{d+1}(g_f(\pi(q_j))+\log\|q_j\|)=:C
\end{gather*}
on $\bP^1$, and in turn on $\sP^1$ by
the density of $\bP^1$ in $\sP^1$ and
the continuity of (the $\exp$ of) both sides on $\sP^1$.
Integrating both sides against $\mu_f$ over $\sP^1$, 
since
$\int_{\sP^1}U_{g_f,[f=\Id_{\bP^1}]}\mu_f=\int_{\sP^1}U_{g_f,\mu_f}[f=\Id_{\bP^1}]=0$
(by $U_{g_f,\mu_f}\equiv 0$),
we have
$C=\int_{\sP^1}\Phi(f,\Id_{\bP^1})_{g_f}\mu_f$, so that
we first have 
\begin{gather*}
 \Phi(f,\Id_{\bP^1})_{g_f}
=U_{g_f,[f=\Id_{\bP^1}]}+\int_{\sP^1}\Phi(f,\Id_{\bP^1})_{g_f}\mu_f\quad\text{on }\sP^1.
\end{gather*}

Fix $z_0\in\bP^1\setminus(\supp[f=\Id_{\bP^1}])$. Using
the above equality twice,
by $f_*[f=\Id_{\bP^1}]=[f=\Id_{\bP^1}]$ on $\sP^1$ and \eqref{eq:Riesz}, we have
\begin{align*}
 &\Phi_{g_f}(f(z_0),z_0)-\int_{\sP^1}\Phi(f,\Id_{\bP^1})_{g_f}\mu_f\\
 =&U_{g_f,[f=\Id_{\bP^1}]}(z_0)
 =U_{g_f,f_*[f=\Id_{\bP^1}]}(z_0)
=\int_{\sP^1}\Phi_{g_f}(z_0,\cdot)(f_*[f=\Id_{\bP^1}])(\cdot)\\
=&\int_{\sP^1}\Phi_{g_f}(z_0,f(\cdot))[f=\Id_{\bP^1}](\cdot)
=\int_{\sP^1}U_{g_f,f^*\delta_{z_0}}[f=\Id_{\bP^1}]\\
=&\int_{\sP^1}U_{g_f,[f=\Id_{\bP^1}]}(f^*\delta_{z_0})
=\int_{\sP^1}\Bigl(\Phi(f,\Id_{\bP^1})_{g_f}
-\int_{\sP^1}\Phi(f,\Id_{\bP^1})_{g_f}\mu_f\Bigr)(f^*\delta_{z_0})\\
=&\int_{\sP^1}\Phi(f,\Id_{\bP^1})_{g_f}(f^*\delta_{z_0})
-d\cdot\int_{\sP^1}\Phi(f,\Id_{\bP^1})_{g_f}\mu_f,
\end{align*}  
and moreover, $\int_{\sP^1}\Phi(f,\Id_{\bP^1})_{g_f}(f^*\delta_{z_0})
 =U_{g_f,f^*\delta_{z_0}}(z_0)=\Phi_{g_f}(f(z_0),z_0)$ by \eqref{eq:Riesz}.
 Hence $(d-1)\int_{\sP^1}\Phi(f,\Id_{\bP^1})_{g_f}\mu_f=0$, and 
in turn since $d>1$, 
\begin{gather}
 \int_{\sP^1}\Phi(f,\Id_{\bP^1})_{g_f}\mu_f=0.\label{eq:vanish}
\end{gather}
 This completes the proof.
\end{proof}

From now on, we focus on the case where $\infty\in\sF(f)$.
We adopt the following convention when no confusion would be caused.
\begin{convention}
 For every probability Radon measure $\nu$ supported by
 $\sP^1\setminus\{\infty\}$,
 we denote $p_{\infty,\nu}$ and $I_{\infty,\nu}$
 by $p_\nu$ and $I_\nu$, respectively,
 for simplicity.
\end{convention}

Since $\supp\mu_f=\sJ(f)\subset\sP^1\setminus D_\infty$,
the equality \eqref{eq:energyinfty} below implies that
$\sP^1\setminus D_\infty$ is of logarithmic capacity $>0$ with pole $\infty$.

\begin{lemma}
\label{th:nonpolar}
Suppose that $\infty\in\sF(f)$. Then
\begin{gather}
 p_{\mu_f}=g_f-\log[\cdot,\infty]_{\can}+\frac{I_{\mu_f}}{2}\quad\text{on }\sP^1,
 \label{eq:continuity}\\
 I_{\mu_f}=-2\cdot g_f(\infty)>-\infty,\text{ and}
\label{eq:energyinfty}\\
 \Phi_{g_f}(\cdot,\infty)
 =-p_{\mu_f}+I_{\mu_f}\quad\text{on }\sP^1.\label{eq:infty}
\end{gather} 
\end{lemma}

\begin{proof}
Suppose $\infty\in\sF(f)$. Then we have
$\supp\mu_f=\sJ(f)\subset\sP^1\setminus D_\infty$ and
\begin{gather*}
 0=V_{g_f}=\int_{\sP^1\times\sP^1}\Phi_{g_f}(\mu_f\times\mu_f)
=I_{\mu_f}-2\cdot\int_{\sP^1}(g_f-\log[\cdot,\infty]_{\can})\mu_f,
\end{gather*}
so that
$I_{\mu_f}=2\cdot\int_{\sP^1}(g_f-\log[\cdot,\infty]_{\can})\mu_f$,
which with 
\begin{gather*}
 0\equiv U_{g_f,\mu_f}=p_{\mu_f}-(g_f-\log[\cdot,\infty]_{\can})-\int_{\sP^1}(g_f-\log[\cdot,\infty]_{\can})\mu_f\quad\text{on }\sP^1 
\end{gather*}
yields \eqref{eq:continuity}.
By \eqref{eq:continuity} and 
$\log[z,\infty]=\log[z,0]-\log|z|$ on $\bP^1\setminus\{\infty\}$,
we have 
\begin{gather*}
 g_f(\infty)=\lim_{z\to\infty}\bigl((p_{\mu_f}(z)-\log|z|)+\log[z,0]\bigr)
-\frac{I_{\mu_f}}{2}=-\frac{I_{\mu_f}}{2}, 
\end{gather*}
so that \eqref{eq:energyinfty} holds.
By \eqref{eq:continuity} and \eqref{eq:energyinfty}, we have
$\Phi_{g_f}(\cdot,\infty)=\log[\cdot,\infty]_{\can}-g_f-g_f(\infty)
=(-p_{\mu_f}+I_{\mu_f}/2)+I_{\mu_f}/2=-p_{\mu_f}+I_{\mu_f}$ on $\sP^1$,
so \eqref{eq:infty} also holds.
\end{proof}

Let $F=(F_0,F_1)\in(K[p_0,p_1]_d)^2$ be a normalized lift of $f$,
and $c_0^F,c_1^F\in K\setminus\{0\}$ be the
coefficients of the maximal degree terms of
$F_0(1,z),F_1(1,z)\in K[z]$, respectively.
No matter whether $\infty\in\sF(f)$,
by the equality $[z,\infty]=1/\|(1,z)\|$ on $\bP^1$ and
the definition of $T_F$, we have
\begin{gather*}
 T_F=-\log[f(\cdot),\infty]_{\can}
+\log|F_0(1,\cdot)|_\infty+d\cdot\log[\cdot,\infty]_{\can}
\end{gather*}
on $\bP^1\setminus(\{\infty\}\cup f^{-1}(\infty))$,
and in turn on $\sP^1\setminus(\{\infty\}\cup f^{-1}(\infty))$
by the density of $\bP^1$ in $\sP^1$ and the continuity of both sides
on $\sP^1\setminus(\{\infty\}\cup f^{-1}(\infty))$.
By \eqref{eq:invarianceGreen}, this equality is rewritten as
\begin{gather}
 d\cdot(g_f-\log[\cdot,\infty]_{\can})-(g_f\circ f-\log[f(\cdot),\infty]_{\can})
 =\log|F_0(1,\cdot)|_{\infty}\label{eq:prepullback}
\end{gather}
on $\sP^1\setminus(\{\infty\}\cup f^{-1}(\infty))$.

\begin{lemma}[Pullback formula for $p_{\mu_f}$ under $f$]
If $\infty\in\sF(f)$, then 
\begin{gather}
 \log|F_0(1,\cdot)|_\infty
 =d\cdot p_{\mu_f}-p_{\mu_f}\circ f-(d-1)\frac{I_{\mu_f}}{2}
\label{eq:invariance}
\end{gather}
on $\sP^1\setminus(\{\infty\}\cup f^{-1}(\infty));$
moreover,
for every $\cS'\in\sP^1\setminus\{\infty,f(\infty)\}$, 
\begin{multline}
p_{\mu_f}(\cS')-\int_{\sP^1\setminus\{\infty\}}p_{\mu_f}(f^*\delta_{\cS'})
+(d-1)I_{\mu_f}\\
=
-\int_{\sP^1}\log|F_0(1,\cdot)|_\infty\frac{f^*\delta_{\cS'}}{d}+(d-1)\frac{I_{\mu_f}}{2},\label{eq:pullbackgeneral}
\end{multline}
and similarly
\begin{gather}
 \int_{\sP^1\setminus\{\infty\}}p_{\mu_f}(f^*\delta_\infty)-(d-1)I_{\mu_f}
=-\log|c_0^F|-(d-1)\frac{I_{\mu_f}}{2}.\label{eq:coeffdenomgeneral}
\end{gather}
\end{lemma}

\begin{proof}
Suppose $\infty\in\sF(f)$. 
Then for every $\cS'\in\sP^1\setminus\{\infty,f(\infty)\}$, 
by \eqref{eq:prepullback} and \eqref{eq:continuity}, 
we have \eqref{eq:invariance}.
Integrating both sides in
\eqref{eq:invariance} against $f^*\delta_{\cS'}/d$
over $\sP^1$, we have \eqref{eq:pullbackgeneral}. Similarly,
integrating both sides in \eqref{eq:invariance} against $\mu_f$ 
over $\sP^1$, also by $f_*\mu_f=\mu_f$
and $I_{\mu_f}:=\int_{\sP^1}p_{\mu_f}\mu_f$, we have
\begin{multline*}
 \log|c_0^F|+\int_{\sP^1\setminus\{\infty\}}p_{\mu_f}(f^*\delta_\infty)
=\int_{\sP^1}\log|F_0(1,\cdot)|_\infty\mu_f\\
=d\cdot I_{\mu_f}-\int_{\sP^1}(p_{\mu_f}\circ f)\mu_f-(d-1)\frac{I_{\mu_f}}{2}
=(d-1)\frac{I_{\mu_f}}{2},
\end{multline*}
so \eqref{eq:coeffdenomgeneral} also holds.
\end{proof}

If $f(\infty)=\infty$, then $F(0,1)=(0,c^F_1)$, so that by the homogeneity of $F$,
for every $n\in\bN$, $F^n(0,1)=(0,(c^F_1)^{(d^n-1)/(d-1)})$ and that
\begin{gather*}
 g_f(\infty)=\lim_{n\to\infty}\frac{T_{F^n}(\infty)}{d^n}
 =\lim_{n\to\infty}\frac{\log\|F^n(0,1)\|}{d^n}-\log\|(0,1)\|
 =\frac{\log|c^F_1|}{d-1}.
\end{gather*}

\begin{lemma}
If $f(\infty)=\infty\in\sF(f)$, then 
\begin{gather}
 I_{\mu_f}=-\frac{2}{d-1}\log|c^F_1|\label{eq:coeffnum} 
\end{gather}
and, for every $\cS'\in\sP^1$,
\begin{gather}
 \int_{\sP^1\setminus\{\infty\}}p_{\mu_f}(f^*\delta_{\cS'})-(d-1)I_{\mu_f}
=\begin{cases}
 p_{\mu_f}(\cS') & \text{if }\cS'\neq\infty,\\
\displaystyle \log\biggl|\frac{c_1^F}{c_0^F}\biggr| & \text{if }\cS'=\infty.
\end{cases}\label{eq:coeffdenomfixed}
\end{gather}
\end{lemma}

\begin{proof}
Suppose that $f(\infty)=\infty\in\sF(f)$. Then
by the above computation of $g_f(\infty)$
and \eqref{eq:energyinfty}, we have \eqref{eq:coeffnum}. 
Moreover, for every $\cS'\in\sP^1\setminus\{\infty\}$, 
using \eqref{eq:infty} twice and \eqref{eq:Riesz}
(and the assumption $f(\infty)=\infty$), we compute
\begin{multline*}
-p_{\mu_f}(\cS')+I_{\mu_f}=\Phi_{g_f}(\infty,\cS')=
 \Phi_{g_f}(f(\infty),\cS')\\
 =\int_{\sP^1}\Phi_{g_f}(\infty,\cdot)(f^*\delta_{\cS'})
 =-\int_{\sP^1}p_{\mu_f}(f^*\delta_{\cS'})+d\cdot I_{\mu_f},
\end{multline*}
so \eqref{eq:coeffdenomfixed} holds for $\cS'\in\sP^1\setminus\{\infty\}$.
Finally, \eqref{eq:coeffdenomfixed} for $\cS'=\infty$
holds by \eqref{eq:coeffdenomgeneral} and \eqref{eq:coeffnum}.
\end{proof}

Let us now focus on $\nu_\infty=\nu_{\infty,\sP^1\setminus D_{\infty}}$
when $\infty\in\sF(f)$.
Then $f(\infty)\in\sF(f)$ and,
since $\supp\nu_\infty\subset\partial D_\infty\subset\sJ(f)=\supp\mu_f$, 
we have 
\begin{gather*}
 \supp(f_*\nu_\infty)\subset f(\sJ(f))=\sJ(f)=\supp\mu_f\subset\sP^1\setminus D_\infty.
\end{gather*}

\begin{lemma}
Suppose that $\infty\in\sF(f)$. 
Then for every $\cS'\in\sP^1\setminus\{\infty,f(\infty)\}$, 
\begin{multline}
 p_{f_*\nu_\infty}(\cS')
 -\int_{\sP^1}p_{\nu_\infty}(f^*\delta_{\cS'})
+d\cdot I_{\nu_\infty}-\int_{\sP^1}(p_{f_*\nu_\infty})\mu_f\\
=p_{\mu_f}(\cS')-\int_{\sP^1}p_{\mu_f}(f^*\delta_{\cS'})+(d-1)I_{\mu_f}
\label{eq:potentialpushforward}
\end{multline}
and, if in addition $\nu_\infty$ is invariant under $f$ in that
$f_*\nu_\infty=\nu_\infty$ on $\sP^1$, then
\begin{multline}
 p_{\nu_\infty}(\cS')-\int_{\sP^1}p_{\nu_\infty}(f^*\delta_{\cS'})
+(d-1)\cdot I_{\nu_\infty}\\
 =p_{\mu_f}(\cS')-\int_{\sP^1}p_{\mu_f}(f^*\delta_{\cS'})+(d-1)I_{\mu_f}.\label{eq:potentialpushforwardinv}
\end{multline}
\end{lemma}

\begin{proof}
Suppose that $\infty\in\sF(f)$. 
Then for every $\cS'\in\sP^1\setminus\{\infty,f(\infty)\}$,
using $\eqref{eq:continuity}$ repeatedly and \eqref{eq:Riesz}, we have
\begin{align*}
&p_{f_*\nu_\infty}(\cS')
=\int_{\sP^1}\log|\cS'-\cdot|_\infty(f_*\nu_\infty)
=\int_{\sP^1}\log|\cS'-f(\cdot)|_\infty\nu_\infty\\
=&\int_{\sP^1}\Bigl(\Phi_{g_f}(f(\cdot),\cS')+\bigl(p_{\mu_f}(f(\cdot))-\frac{I_{\mu_f}}{2}\bigr)+\bigl(p_{\mu_f}(\cS')-\frac{I_{\mu_f}}{2}\bigr)\Bigr)\nu_\infty\\
=&\int_{\sP^1}\Bigl(\int_{\sP^1}\Phi_{g_f}(\cdot,\cS)(f^*\delta_{\cS'})(\cS)\Bigr)\nu_\infty
+\int_{\sP^1}(p_{\mu_f}\circ f)\nu_\infty+p_{\mu_f}(\cS')-I_{\mu_f}\\
=&\int_{\sP^1}\biggl(\int_{\sP^1}\Bigl(\log|\cS-\cdot|_\infty-\bigl(p_{\mu_f}(\cS)-\frac{I_{\mu_f}}{2}\bigr)-\bigl(p_{\mu_f}(\cdot)-\frac{I_{\mu_f}}{2}\bigr)\Bigr)(f^*\delta_{\cS'})(\cS)\biggr)\nu_\infty\\
&\quad+\int_{\sP^1}(p_{\mu_f}\circ f)\nu_\infty+p_{\mu_f}(\cS')-I_{\mu_f}\\
=&\int_{\sP^1}p_{\nu_\infty}(f^*\delta_{\cS'})
+\int_{\sP^1}(p_{\mu_f}\circ f-d\cdot p_{\mu_f})\nu_\infty\\
&\quad+p_{\mu_f}(\cS')-\int_{\sP^1}p_{\mu_f}(f^*\delta_{\cS'})+(d-1)I_{\mu_f}.
\end{align*}
Moreover, by Fubini's theorem and $p_{\nu_\infty}\equiv I_{\nu_\infty}$
on $\sP^1\setminus D_\infty$, we also have 
\begin{multline*}
\int_{\sP^1}(p_{\mu_f}\circ f-d\cdot p_{\mu_f})\nu_\infty\\
=\int_{\sP^1}p_{\mu_f}(f_*\nu_\infty)-d\cdot\int_{\sP^1}p_{\mu_f}\nu_\infty
=\int_{\sP^1}(p_{f_*\nu_\infty})\mu_f-d\cdot I_{\nu_\infty},
\end{multline*}
which completes the proof of \eqref{eq:potentialpushforward}.

If in addition $f_*\nu_\infty=\nu_\infty$ on $\sP^1$, then
by the identity $p_{\nu_\infty}\equiv I_{\nu_\infty}$
on $\sP^1\setminus(D_\infty\cup E)$, where $E$ is an $F_\sigma$-subset
in $\partial D_\infty$ of logarithmic capacity $0$ with pole
$\infty$, and by the vanishing $\mu_f(E)=0$, we also have
\begin{gather}
 \int_{\sP^1}(p_{f_*\nu_\infty})\mu_f=
\int_{\sP^1}(p_{\nu_\infty})\mu_f=I_{\nu_\infty},\label{eq:Fubinienergy}
\end{gather}
which completes the proof of \eqref{eq:potentialpushforwardinv}.
\end{proof}

\begin{lemma}[Invariance of $\nu_\infty$ under $f$]\label{th:invharmonic}
If $f(\infty)=\infty\in\sF(f)$, then $f_*\nu_\infty=\nu_\infty$ on $\sP^1$
and, for every $\cS'\in\sP^1$,
\begin{gather}
\int_{\sP^1\setminus\{\infty\}}p_{\nu_\infty}(f^*\delta_{\cS'})-(d-1)I_{\nu_\infty}
=\begin{cases}
  p_{\nu_\infty}(\cS') & \text{if }\cS'\neq\infty,\\
\displaystyle  \log\biggl|\frac{c^F_1}{c^F_0}\biggr| & \text{if }\cS'=\infty.
 \end{cases}
\label{eq:Rieszharmonic}
\end{gather}
\end{lemma}

\begin{proof}
Suppose that $f(\infty)=\infty\in\sF(f)$. Then 
for every $\cS'\in\sP^1\setminus\{\infty\}$, 
by \eqref{eq:potentialpushforward} and \eqref{eq:coeffdenomfixed}, we have
\begin{gather}
  p_{f_*\nu_\infty}(\cS')
 =\int_{\sP^1}p_{\nu_\infty}(f^*\delta_{\cS'})
 -d\cdot I_{\nu_\infty}+\int_{\sP^1}(p_{f_*\nu_\infty})\mu_f.
\tag{\ref{eq:potentialpushforward}$'$} \label{eq:potentialpushforwardfixed}
\end{gather}
We claim that
\begin{gather}
 p_{f_*\nu_\infty}\equiv
\int_{\sP^1}(p_{f_*\nu_\infty})\mu_f\quad\text{on }\sJ(f);\label{eq:constant}
\end{gather}
for, by the equality \eqref{eq:potentialpushforwardfixed} and
$p_{\nu_\infty}\ge I_{\nu_\infty}$ on $\sP^1$
(and Fubini's theorem and \eqref{eq:continuity}), we have
\begin{gather*}
p_{f_*\nu_\infty}\ge
\int_{\sP^1}(p_{f_*\nu_\infty})\mu_f
>-\infty\quad\text{on }\sP^1\setminus\{\infty\},
\end{gather*}
so that $p_{f_*\nu_\infty}\equiv\int_{\sP^1}p_{\mu_f}(f_*\nu_\infty)$
$\mu_f$-a.e.\ on $\sP^1$. Hence the claim follows by 
the strong upper semicontinuity \eqref{eq:susc} of $p_{f_*\nu_\infty}$
on $\sP^1$ and $\sJ(f)=\supp\mu_f$, also recalling Remark \ref{th:potgood}.

Once the identity \eqref{eq:constant} is at our disposal, using also 
the maximum principle for the subharmonic function $p_{f_*\nu_\infty}$
and the latter inequality in \eqref{eq:non-equilibrium}, we have
\begin{gather*}
p_{f_*\nu_\infty}
\equiv\int_{\sP^1}(p_{f_*\nu_\infty})\mu_f
=\sup_{\sJ(f)}p_{f_*\nu_\infty}\ge
\sup_{\sP^1\setminus D_\infty}p_{f_*\nu_\infty}\ge
I_{\nu_\infty}\quad\text{on }\sJ(f),
\end{gather*}
and integrating
both sides of this inequality against $f_*\nu_\infty$, we have 
$I_{f_*\nu_\infty}\ge I_{\nu_\infty}$ or equivalently 
\begin{gather*}
 f_*\nu_\infty=\nu_\infty\quad\text{on }\sP^1. 
\end{gather*}
Then 
\eqref{eq:Rieszharmonic} 
holds for every $\cS'\in\sP^1\setminus\{\infty\}$
by \eqref{eq:potentialpushforwardinv} and \eqref{eq:coeffdenomfixed}.
Finally, 
integrating both sides in \eqref{eq:invariance} against $\nu_\infty$ 
over $\sP^1$, by \eqref{eq:Fubinienergy} and Fubini's theorem, we compute
\begin{multline*}
\log|c_0^F|+\int_{\sP^1\setminus\{\infty\}}p_{\nu_\infty}(f^*\delta_\infty)
=\int_{\sP^1}\log|F_0(1,\cdot)|_\infty\nu_\infty\\
=d\cdot I_{\nu_\infty}
-\int_{\sP^1}(p_{\mu_f}\circ f)\nu_\infty
-(d-1)\frac{I_{\mu_f}}{2}\\
=d\cdot I_{\nu_\infty}-\int_{\sP^1}(p_{f_*\nu_\infty})\mu_f
-(d-1)\frac{I_{\mu_f}}{2}
=(d-1)I_{\nu_\infty}-(d-1)\frac{I_{\mu_f}}{2},
\end{multline*}
which with \eqref{eq:coeffnum} yields \eqref{eq:Rieszharmonic} for $\cS'=\infty$.
\end{proof}

\begin{remark}
 All the computations in this Section are also valid for $K=\bC$.
\end{remark}

\begin{remark}
The $f$-invariance of $\nu_{\infty}$ in Lemma \ref{th:invharmonic}
is a non-archimedean counterpart to Ma\~n\'e--da Rocha \cite[p.253, before Corollary 1]{MR92}. Their argument was based on solving Dirichlet problem 
using the Poisson kernel on $D_\infty\cup\partial D_\infty$.
A similar machinery has been only partly developed 
in the potential theory on $\sP^1$ (see \cite[\S 7.3, \S7.6]{BR10}).
\end{remark}

\section{Proof of Theorem \ref{th:polynomial}}\label{sec:proof}

Let $K$ be an algebraically closed field that is
complete with respect to a non-trivial and non-archimedean
absolute value $|\cdot|$.
Let $f\in K(z)$ be a rational function of degree $d>1$, and
$F=(F_0,F_1)\in(K[p_0,p_1]_d)^2$ be a normalized lift of $f$.
When $\infty\in\sF(f)$, let us also denote 
$\nu_{\sP^1\setminus D_\infty}=\nu_{\infty,\sP^1\setminus D_\infty}$
by $\nu_\infty$ for simplicity.
If $\mu_f=\nu_\infty$ on $\sP^1$, then not only
$p_{\mu_f}=p_{\nu_\infty}>I_{\nu_\infty}=I_{\mu_f}$ 
on $D_\infty$ but, by the continuity of $p_{\mu_f}$ on
$\sP^1\setminus\{\infty\}$ (by \eqref{eq:continuity}), also
$p_{\mu_f}=p_{\nu_\infty}\equiv I_{\nu_\infty}=I_{\mu_f}$ 
on $\sP^1\setminus D_\infty$. 

Suppose that $\infty\in\sF(f)$, $f(D_\infty)=D_\infty$
(so $D_\infty\subset f^{-1}(D_\infty)$), and $\mu_f=\nu_\infty$ on $\sP^1$.
Then by \eqref{eq:invariance} and 
$p_{\mu_f}\equiv I_{\mu_f}$ on $\sP^1\setminus D_\infty$, we have
\begin{gather}
 \log|F_0(1,\cdot)|_\infty\equiv (d-1)\frac{I_{\mu_f}}{2}=:I_0
\quad\text{on }\sP^1\setminus f^{-1}(D_\infty).\label{eq:denomconst}
\end{gather}
Let $\cS_0$ be the point in $\sH^1$ represented by 
the disk $\{z\in K:|z|\le e^{I_0}\}$ in $K$.

Suppose also that $f^{-1}(D_\infty)\setminus D_\infty\neq\emptyset$.
Then $\deg F_0(1,z)>0$. The subset
\begin{gather*}
 U_\infty:=\{\cS\in\sP^1:|F_0(1,\cS)|_\infty>e^{I_0}\}
\end{gather*}
in $\sP^1$ is the component of $\sP^1\setminus(F_0(1,\cdot))^{-1}(\cS_0)$ 
containing $\infty$, and $\partial U_\infty=(F_0(1,\cdot))^{-1}(\cS_0)$. 
By \eqref{eq:denomconst}, we have $U_\infty\subset f^{-1}(D_\infty)$, 
and in turn
\begin{gather*}
 U_\infty\subset D_\infty.
\end{gather*}
For every $w\in f^{-1}(\infty)\setminus\{\infty\}
=(F_0(1,\cdot))^{-1}(0)\subset\{\cS\in\sP^1:|F_0(1,\cS)|_\infty<e^{I_0}\}$, 
let $D_w$ (resp.\ $U_w$)
be the component of $f^{-1}(D_\infty)$ 
(resp.\ the component of $\{\cS\in\sP^1:|F_0(1,\cS)|_\infty<e^{I_0}\}$)
containing $w$. 
Then 
$U_w$ is the component of $\sP^1\setminus(F_0(1,\cdot))^{-1}(\cS_0)$
containing $w$, and 
$\partial U_w$ is a singleton in 
$(F_0(1,\cdot))^{-1}(\cS_0)=\partial U_\infty$.

We claim that $\partial D_\infty$ is a singleton say $\{\cS_\infty\}$ in $\sH^1$ and, moreover, 
that for every $w\in f^{-1}(\infty)\setminus D_\infty
(\neq\emptyset$ under the assumption that $f^{-1}(D_\infty)\setminus D_\infty\neq\emptyset$), 
\begin{gather*}
 \partial D_w=\partial D_\infty(=\{\cS_\infty\});
\end{gather*}
indeed, for every $w\in f^{-1}(\infty)\setminus D_\infty$,
we not only have $D_w\subset U_w$
(since otherwise, we must have $\emptyset\neq D_w\cap U_\infty\subset D_w\cap D_\infty$
so $D_w=D_\infty$, which contradicts $w\not\in D_\infty$) but also
$U_w\subset D_w$ (by \eqref{eq:denomconst}), so that $U_w=D_w$.
This together with $\partial U_w\subset\partial U_\infty$ and 
$U_\infty\subset D_\infty$ yields
\begin{gather*}
\partial D_w=\partial U_w\subset\partial D_\infty 
\end{gather*}
(since otherwise, we must have $\emptyset\neq U_w\cap D_\infty=D_w\cap D_\infty$ so $D_w=D_\infty$, which contradicts $w\not\in D_\infty$).
Hence the claim holds since
$f(\partial U_w)=f(\partial D_w)=\partial D_\infty$ is a singleton in $\sH^1$.

Now fix $w\in f^{-1}(\infty)\setminus D_\infty=f^{-1}(\infty)\setminus 
f(D_\infty)$, and recall that
$\mu_{f^2}=\mu_f$ on $\sP^1$ (so $\sF(f^2)=\sF(f)\ni\infty$ and $D_\infty(f^2)=D_\infty(f)$).
Applying the above argument to $f^2$ and
every $w'\in f^{-1}(w)\subset f^{-2}(\infty)\setminus D_\infty$, we have 
\begin{gather*}
 \partial D^{(2)}_{w'}=\{\cS_\infty\},
\end{gather*}
where $D^{(2)}_{w'}$ is the component of $f^{-2}(D_\infty)$ containing $w'$
or equivalently the component of $f^{-1}(D_w)$ containing $w'$.
Consequently, 
\begin{gather*}
 f^{-1}(\cS_\infty)=f^{-1}(\partial D_w)
\subset\bigcup_{w'\in f^{-1}(w)}\partial D^{(2)}_{w'} =\{\cS_\infty\},
\end{gather*}
so $f$ has a potential good reduction. \qed

\section{Proof of Theorem \ref{th:counterexample}}
\label{sec:counterexample}

Pick a prime number $p$, and let us denote $|\cdot|_p$ by $|\cdot|$ for simplicity.
Set
\begin{gather*}
f(z):=\frac{z^p-z}{p}\in\bQ[z]\quad\text{and}\quad
A(z):=\frac{az+b}{cz+d}\in\PGL(2,\bZ_p).
\end{gather*} 
If $|c|<1$, then $|ad-bc|=|ad|=1$, so that $|a|=|d|=1$. 

Let $\sJ(f\circ A)$ and $\sF(f\circ A)$
denote the Berkovich Julia and Fatou sets in $\sP^1(\bC_p)$
of $f\circ A$ as an element of $\bC_p(z)$ of degree $p$, 
respectively.

\subsection{Computing $\sJ(f\circ A)$}

\begin{lemma}\label{th:invariance}
If $|c|<1$, then $(f\circ A)^{-1}(\bZ_p)=\bZ_p$. 
\end{lemma}

\begin{proof}
We first claim that for every $z \in \mathbb{Z}$, $p\cdot f(z)=z^p-z\equiv 0$ 
modulo $p\bZ$; indeed, when is obvious if $z=0$ modulo $p\bZ$, and is the case
by Fermat's Little Theorem when $z\neq 0$ modulo $p\bZ$. By this claim,
we have $f(\bZ)\subset\bZ$ 
(cf.\ \cite{WoodcockSmart98}), and in turn $f(\bZ_p)\subset\bZ_p$
by the continuity of the action of $f$ on $\bQ_p$ 
and the density of $\bZ$ in $\bZ_p$.
Next, we claim that $f^{-1}(\bZ_p)\subset\bZ_p$ or equivalently that
for every $w\in\bZ_p$, $f^{-1}(w)\subset\bZ_p$; 
indeed, setting $W(X):=X^p-X-pw\in\bZ_p[X]$ of degree $p$,
we have already seen that
the reduction $\overline{W}(X)=X^p-X\in\bF_p[X]$ of $W$ modulo $p\bZ_p$
has $p$ distinct roots $\overline{0},\ldots,\overline{p-1}$ in $\bF_p$.
Hence by Hensel's lemma (see, e.g., \cite[\S3.3.4, Proposition 3]{BGR}), 
$W(X)$ also has $p$ distinct roots in $\bZ_p$, and has no other roots in $\overline{\bQ_p}$, so the claim holds. 
We have seen that $f^{-1}(\bZ_p)=\bZ_p$. 

Suppose now that $|c|<1$. Then for every $z\in\bZ_p$, 
we have $|cz|<1=|d|$, so that $|A(z)|=|az+b|/|cz+d|=|az+b|\le 1$.
Hence $A(\bZ_p)\subset\bZ_p$, and similarly $A^{-1}(\bZ_p)\subset\bZ_p$
since $A^{-1}(z)=(dz-b)/(-cz+a)\in\PGL(2,\bZ_p)$ and $|-c|=|c|<1$. 
Now we conclude that $(f\circ A)^{-1}(\bZ_p)=A^{-1}(\bZ_p)=\bZ_p$.
\end{proof}

\begin{lemma}\label{th:implicit}
 If $|b|\ll 1$ and $|c|\ll 1$, then $f\circ A$ has an attracting fixed point 
 $z_A$ in $\bP^1(\bC_p)\setminus\bZ_p$, which tends to $\infty$ 
 as $(a,b,c,d)\to(1,0,0,1)$ in $(\bZ_p)^4$.
 Moreover, if in addition $c\neq 0$, then $z_A\in\bC_p\setminus\bZ_p$
 and $(f\circ A)^{-1}(z_A)\neq\{z_A\}$.
\end{lemma}

\begin{proof}
 Since $f^{-1}(\infty)=\{\infty\}$ and $\deg f=p>1$, 
 the former assertion holds 
also
 noting that $(\Id_{\bP^1(\bC_p)})'\equiv 1\neq 0$ and applying
 an implicit function theorem to the equation $(f\circ A)(z)=z$
 near $(z,a,b,c,d)=(\infty,1,0,0,1)$ in $\bP^1(\bC_p)\times(\bZ_p)^4$
 (see, e.g., \cite[(10.8)]{Abhyankar64}). 
 Moreover, since 
 $f'(z)=z^{p-1}-p^{-1}$ and $f''(z)=(p-1)z^{p-2}$, 
 the point $A^{-1}(\infty)=-d/c$ is the unique point $z\in\bP^1(\bC_p)$ 
 such that $\deg_z(f\circ A)=p(=\deg(f\circ A))$, and on the other hand, 
 if in addition $c\neq 0$, then the point $A^{-1}(\infty)$
 is $\neq\infty$ and
 is not fixed by $f\circ A$.
 Hence the latter assertion holds also noting that
 $(f\circ A)(\infty)\neq\infty$ if in addition $c\neq 0$.
\end{proof}

Consequently, if $|b|\ll 1$ and $|c|\ll 1$, then
\begin{gather}
\sJ(f\circ A)=\bZ_p=\sP^1(\bC_p)\setminus D_{z_A}(f\circ A);\label{eq:Julia}
\end{gather}
indeed, by Lemma \ref{th:invariance} (and \eqref{eq:equidist}),
if $|c|<1$, then $\sJ(f\circ A)\subset\bZ_p$.
If in addition $|b|\ll 1$ and $|c|\ll 1$, then 
by Lemma \ref{th:implicit} (and $\bZ_p\subset\bC_p$), 
we have $\sF(f\circ A)=D_{z_A}(f\circ A)$, 
which is an (immediate) attractive basin of $f$ 
(see \cite[Th\'eor\`eme de Classification]{Juan03}) 
associated with $z_A\in\bP^1(\bC_p)\setminus\bZ_p$, 
and in turn have $\sJ(f\circ A)=\bZ_p$
since $(f\circ A)(\bZ_p)\subset\bZ_p$ by Lemma \ref{th:invariance}.

\subsection{Computing energies and measures}
Since 
\begin{gather*}
 \Res\bigl(p^{1/2}\cdot\bigl(z_0^p,z_0^pf(z_1/z_0)\bigr)\bigr)
 =(p^{1/2})^{2p}\cdot(1^{p-p}\cdot(p^{-1})^{p-0}\cdot 1)=1,
\end{gather*}
the pair 
\begin{gather*}
 F(z_0,z_1):=p^{1/2}\cdot\bigl(z_0^p,
z_0^pf(z_1/z_0)\bigr)\in(\bQ[z_0,z_1]_p)^2 
\end{gather*}
is a normalized lift of $f$. 
Noting that $|\Res(az_0+bz_1,cz_0+dz_1)|=|ad-bc|=1$ and
using a formula for the homogeneous resultant of the composition
of homogeneous polynomial maps
(see, e.g., \cite[Exercise 2.12]{SilvermanDynamics}), we also have
$\bigl|\Res\bigl(F(az_0+bz_1,cz_0+dz_1)\bigr)\bigr|
 =\bigl|(\Res F)^1\cdot(\Res(az_0+bz_1,cz_0+dz_1))^p\bigr|
 =1$,
so that 
\begin{multline*}
 F_A(z_0,z_1)
:=F(az_0+bz_1,cz_0+dz_1)\\
=p^{1/2}\cdot\biggl((az_0+bz_1)^p,
\frac{(cz_0+dz_1)^p-(az_0+bz_1)^{p-1}(cz_0+dz_1)}{p}\biggr)
\in(\bQ_p[z_0,z_1]_p)^2
\end{multline*} 
is a normalized lift of $f\circ A$. For every $n\in\bN$, write 
\begin{gather*}
 (F_A)^n=\bigl(F_{A,0}^{(n)},F_{A,1}^{(n)}\bigr)\in(\bQ_p[z_0,z_1]_{p^n})^2.
\end{gather*}

\begin{lemma}\label{th:infty}
If $|b|<1$ and $|c|<1$, then 
\begin{gather*}
 g_{f\circ A}(\infty)
\biggl(=\sum_{j=1}^\infty\Bigl(\frac{\log\|(F_A)^j(0,1)\|}{p^j}-\frac{\log\|(F_A)^{j-1}(0,1)\|}{p^{j-1}}\Bigr)\biggr)
=\frac{\log p}{2(p-1)}.
\end{gather*}
\end{lemma}

\begin{proof}
Suppose that $|b|<1$ and $|c|<1$(, and recall $|p|=p^{-1}<1$).
Then for every $(z_0,z_1)\in\bC_p^2$, if $|z_0|<|z_1|$, then
\begin{gather*}
 |cz_0+dz_1|=|dz_1|=|z_1|>\max\{|az_0|,|bz_1|\}\ge|az_0+bz_1|
\end{gather*}
so
\begin{gather*}
|F_{A,0}^{(1)}(z_0,z_1)|<|F_{A,1}^{(1)}(z_0,z_1)|\quad\text{and}\\
\begin{aligned}
 \|F_A(z_0,z_1)\|=&|F_{A,1}^{(1)}(z_0,z_1)|=p^{1/2}|cz_0+dz_1|^p\\
 =&p^{1/2}|dz_1|^p=p^{1/2}|z_1|^p
 =p^{1/2}\|(z_0,z_1)\|^p.
\end{aligned}
\end{gather*}
Hence inductively, for every $n\in\bN$, we have
$|F_{A,0}^{(n)}(0,1)|<|F_{A,1}^{(n)}(0,1)|$, and moreover
\begin{multline*}
\sum_{j=1}^n\Bigl(\frac{\log\|(F_A)^j(0,1)\|}{p^j}-\frac{\log\|(F_A)^{j-1}(0,1)\|}{p^{j-1}}\Bigr)
=\sum_{j=1}^n\frac{\frac{1}{2}\log p}{p^j}\\
=\Bigl(\frac{1}{2}\log p\Bigr)\frac{(1/p)(1-1/p^n)}{1-1/p}\to 
\Bigl(\frac{1}{2}\log p\Bigr)\frac{1}{p-1}
\end{multline*}
as $n\to\infty$.
\end{proof}

\begin{lemma}
 If $(a,b,c,d)$ is close enough to $(1,0,0,1)$ in $(\bZ_p)^4$, then 
 \begin{gather*}
 \mu_{f\circ A}=\nu_{\infty,\bZ_p}=\nu_{z_A,\bZ_p}\quad\text{on }\sP^1(\bC_p).
 \end{gather*}
\end{lemma}
\begin{proof}
If $|b|\ll 1$ and $|c|\ll 1$, then by \eqref{eq:Julia} and $\bZ_p\subset\bC_p$,
we have 
\begin{gather*}
 \infty\in\sF(f\circ A)=D_{z_A}(f\circ A)=\sP^1(\bC_p)\setminus\bZ_p. 
\end{gather*}
Then
by \eqref{eq:energyinfty} and Lemma \ref{th:infty}, we have
\begin{gather*}
 I_{\infty,\mu_{f\circ A}}=-2\cdot\biggl(\frac{\log p}{2(p-1)}\biggr)
 =\log p^{\frac{-1}{p-1}},
\end{gather*}
and in particular,
recalling $\nu_{\infty,\bZ_p}=\mu_f$ on $\sP^1(\bC_p)$, also
$I_{\infty,\nu_{\infty,\bZ_p}}=I_{\infty,\mu_f}=\log p^{\frac{-1}{p-1}}$ 
(for a non-dynamical and more direct computation of 
$I_{\infty,\nu_{\infty,\bZ_p}}$, see \cite{BakerHsia05}).
 Now the first equality holds by the uniqueness of the equilibrium mass distribution
 on the non-polar compact subset $\bZ_p$ in $\sP^1(\bC_p)$. The second equality 
 holds since $z_A$ tends to $\infty$ as $(a,b,c,d)\to(1,0,0,1)$
 in $(\bZ_p)^4$ (by Lemma \ref{th:implicit}), also
 recalling Observation \ref{th:perburb}.
\end{proof}

\begin{remark}\label{th:easy}
 If $0<|c|\ll 1$ and $|b|\ll 1$, then 
 $(f\circ A)(\infty)\neq\infty\in\sF(f\circ A)$, 
 $(f\circ A)(D_\infty(f\circ A))=D_\infty(f\circ A)$,
 $\sJ(f\circ A)\not\subset\sH^1$ (indeed $\sJ(f\circ A)\subset\bC_p$), and 
 $\mu_{f\circ A}=\nu_{\infty,\sP^1\setminus D_\infty}$ on $\sP^1$.
\end{remark}

\subsection{Conclusion}
If $|b|\ll 1$ and $0<|c|\ll 1$, then setting 
$m_A(z):=\frac{1}{z-z_A}\in\PGL(2,\bC_p)$, 
the rational function
\begin{gather*}
 g_A:=m_A\circ(f\circ A)\circ m_A^{-1}\in\bC_p(z) 
\end{gather*}
is of degree $p$ and satisfies
$g_A(\infty)=\infty,|g_A'(\infty)|<1, g_A^{-1}(\infty)\neq\{\infty\}$, and
$\infty\in m_A(D_{z_A}(f\circ A))=D_\infty(g_A)$.
If moreover $(a,b,c,d)$ is close enough to $(1,0,0,1)$ in $(\bZ_p)^4$, then 
also recalling Observations \ref{th:affine} and \ref{th:involution},
we have
\begin{multline*}
 \mu_{g_A}=(m_A)_*\mu_{f\circ A}
=(m_A)_*\nu_{\infty,\bZ_p}=(m_A)_*\nu_{z_A,\bZ_p}\\
=(m_A)_*\nu_{z_A,\sP^1\setminus D_{z_A}(f\circ A)}
=\nu_{\infty,\sP^1\setminus D_{\infty}(g_A)}
\quad\text{on }\sP^1(\bC_p).
\end{multline*}
Now the proof of Theorem \ref{th:counterexample} is complete. \qed

\begin{acknowledgement}
The first author was partially supported by JSPS Grant-in-Aid 
for Scientific Research (C), 15K04924 and 19K03541.
\end{acknowledgement} 

\def\cprime{$'$}

\end{document}